\documentclass[reqno,twoside,final]{amsart}
\usepackage{amssymb,amsmath,amscd}
\let\epsilon=\varepsilon
\let\phi=\varphi
\def\({\left(}
\def\){\right)}

\def\dist{{\operatorname{dist}}}
\def\e{{\varepsilon}}
\newtheorem{theorem}{Theorem}

\newtheorem{Lem}{Lemma}
\newtheorem{Cor}{Corollary}

\theoremstyle{definition}

\def\d{{\partial}}
\newcommand{\M}{\mathfrak M}

\newcommand{\R}{\mathbb R}
\newcommand{\di}{\mathrm{d}}

\newcommand{\Bilip}{\mathrm{Bilip}}

\def\mitrii#1{}

\begin{document}

\title{A direct proof of one Gromov's theorem}
\author{Yu.~D.~Burago, S.~G.~Malev, D.~I.~Novikov$^\ddag$}


\address{Steklov Mathematical Institute at St.Petersburg,
Russia} \email{yuburago@pdmi.ras.ru}
\address{Department of Mathematics, Weizmann Institute of Science,
Rehovot, Israel} \email{sergey.malev@weizmann.ac.il}
\address{Department of Mathematics, Weizmann Institute of Science,
Rehovot, Israel} \email{dmitry.novikov@weizmann.ac.il}
\thanks{The first author is partially supported by RFBR grant 05-01-00939}
\thanks{The third author is supported by Samuel M. Soref \& Helene K. Soref Foundation}

\begin{abstract}
We give a new  proof of the Gromov theorem:  For any $C>0$ and integer $n>1$
there exists a function
$\Delta_{C,n}$ such that if the
Gromov--Hausdorff distance between complete Riemannian $n$-manifolds $V$ and $W$ is
not greater than $\delta$, absolute values of 
their
sectional curvatures $|K_{\sigma}|\leq C$, and their
injectivity radii $\geq 1/C$, then
the Lipschitz distance between $V$ and $W$  is less than
$\Delta_{C,n}(\delta)$ and  $\Delta_{C,n}\to 0$ as $\delta\to 0$.
\end{abstract}
 \maketitle

\section{Introduction}

Denote by  $\M(\rho,C,n)$ the  class of complete
 $n$-dimensional
Riemannian manifolds $V$  with   section curvatures
$|K_\sigma|\leq C<\infty$ and the injective radii  $r_{in}(V)\geq \rho$,
where $C$, $\rho$ are some positive constants.
 There are two
well-known metrics on this class: the  Lipschitz metric and the
Gromov--Hausdorff metric.

Recall that the  Lipschitz distance $\di_{Lip}(X,Y)$ between  metric spaces $X$, $Y$
is defined as
\[
\di_{Lip}(V,W)=\ln\inf\{k:\Bilip_k(V,W)\neq\emptyset\},
\]
where $\Bilip_k(V,W)$ denotes the class of all bi-Lipschitz
homeomorphisms between $V$ and $W$ with bi-Lipschitz constant
$k\ge1$. By bi-Lipschitz constant of a homeomorphism $\zeta$
we mean maximum of Lipschitz constants for  $\zeta$ and
$\zeta^{-1}$.

Instead of the Gromov--Hausdorff metric, we  use a metric,
equivalent to it (see, for instance, \cite{BBI}). We preserve
notation $\di_{GH}$ for this metric. By definition,  the distance
$\di_{GH}(V,W)$  is  the infimum of all $\delta>0$ with the
property that there exists a mapping $\chi:V\rightarrow W$ such
that  $\chi(V)$ is a $\delta$-net in $W$ and  $\chi$ changes
distances by at most on $\delta$:
\[
\left|\di_W(\chi(x),\chi(y))-\di_V(x,y)\right|<\delta
\]
for  any points $x,y\in V$. Note that $\chi$ is not supposed to
be continuous.

The purpose of this paper is to give a direct proof for the following Gromov theorem.

\begin{theorem}[Gromov \cite{Gromov}, page 379]\label{thm:main}
For  given $\rho>0$, $C>0$ and an integer $n>1$, there exists a positive function
$\Delta=\Delta_{(C,n,\rho)}$ such that $\Delta(\delta)\to 0$ as $\delta\to +0$ and if
$V,W\in\M(\rho,C,n)$ satisfy the condition $\di_{GH}(V,W)<\delta$ then
$$
\di_{Lip}(V,W)<\Delta(\delta).
$$
\end{theorem}

 In contrast to Gromov's proof using axillary embeddings of the manifolds
 $V$, $W$ into an Euclidean space of a large dimension, we directly construct
a bi-Lipschitz diffeomorphism  $h(x)$ between $V$ and $W$ with a required
bi-Lipschitz constant $\Delta(\delta)$. The map $h(x)$
is obtained by gluing together ``local maps'' $\varphi_i$ with help of
partition of unity. The maps $\varphi_i$ are defined on some balls
$B_{2\e}(v_i)\subset V$
which form a locally finite covering of $V$. This gluing is based on Karcher's
center of mass
technique \cite{Karcher}. The resulting map turns out to be  bi-Lipschitz with the
required constant since the
mappings $\varphi_i$ are $C^1$-close one to another on the intersection of their
domains.

To justify publishing  our proof, note that though ideas of Gromov's
proof explained very clear in his book, some details are
 omitted in his exposition.

Later on  $C$ denotes different constants
depending on $n=\dim V=\dim W$ only. We always assume
$\delta$ to be sufficiently small, $\delta<\delta_0$, where
$\delta_0$ depends on $n$ only. All these constants can be  computed explicitly,
if such a need arises.

\section{Preparations}

By a suitable rescaling of $V,W$, one can get rid of one parameter
and assume that the absolute values of section curvatures smaller than
$\delta$ and the injectivity radii are bigger than $\delta^{-1}$.
Also we can assume that and
$d_{GH}<\delta$.

We always suppose that $0<\delta\ll 1$ and denote
\begin{equation}\label{eq:epsilon}
\e^2=\delta\ll 1.
\end{equation}

\subsection{$\e$-Orthonormal base}
We say that a basis $\{e_1,\dots ,e_n\}\subset\R^n$ is $\e$-orthonormal
if $|(e_i, e_j)-\delta_{ij}|<\e$ for all $i,j=1,\dots,n$, where
$\delta_{ij}$ is the Kronecker symbol. A linear map $L\colon M\to N$ of two
Euclidean spaces is $\e$-close to isometry if $\|L-Q\|<\e$ for some
isometry $Q\colon M\to N$.

\begin{Lem}\label{lem:e-orthonormal}
Let $\{\xi_i\}$ be an $\e$-orthonormal base of $\R^n$,
$\e<\frac 1 {2n}$.
Let $L:\R^n\to\R^n$ be a linear operator.

If $\|L(\xi_i)\|<\delta$, $i=1,\dots,n$ then $\|L\|<2\sqrt{n}\delta$.

Assume that $8n\sqrt{n}\delta<1$ and
$$
|\langle L\xi_i,L\xi_j\rangle-\langle\xi_i,\xi_j\rangle|<\delta,
\quad i,j=1,..,n.
$$
Then $L$ is $8n\sqrt{n}\delta$-close to an isometry.
\end{Lem}
\begin{proof}
Lemma could be easily proved by straightforward  calculation.
We give a proof to the first part, the second part can be obtained  the same way.

Let $x=\sum x_i\xi_i$ be a unit vector. Then
$$
1=\langle x,x\rangle=\sum_{i,j}x_ix_j\langle\xi_i,\xi_j\rangle
\ge\sum_{i=1}^nx_i^2-\varepsilon\left(\sum_{i=1}^n|x_i|\right)^2
\ge(1-n\varepsilon)\sum_{i=1}^nx_i^2.
$$
Therefore $\sum x_i^2\le(1-\varepsilon n)^{-1}\le 2.$ Thus
$$
\|Lx\|=\|\sum_{i=1}^nx_iL(\xi_i)\|\le\delta\sum_{i=1}^n|x_i|
\le\delta \sqrt{n\sum x_i^2}\le 2\delta\sqrt n.
$$
\end{proof}

\subsection{On comparison theorems and exponential mappings.}
Denote by
$J$ a Jacobi vector field
along a geodesic $\gamma\colon [0,2]\to V$. As usually,
$\dot J(t)=\frac{D}{dt}J$ means the covariant derivative of $J$ along $\gamma$.
We need a known comparison theorem of Rauch-style, see
\cite{chavel}, 7.4, or the original
Karcher's paper \cite{Karcher}.  We
formulate the theorem in the form adopted to our case.

\begin{theorem}\label{thm:rauch original}
Suppose that all section curvatures  $|K|\le \delta$ and
$0<\delta\ll 1$ (it is enough to have $\delta<10^{-2}$).

 Then
\begin{multline}\label{eq:Rauch Jacobi}
|J(0)| \cos(\sqrt{\delta}t)+|\dot
J(0)|\delta^{-1/2}\sin(\sqrt{\delta}t) \le |J(t)|\\ \le |J(0)|
\cosh(\sqrt{\delta}t)+|\dot J(0)|\delta^{-1/2}\sinh(\sqrt{\delta}t).
\end{multline}
\end{theorem}

\medskip
As usually, we apply this estimate to the exponential and logarithmic  mappings.
The latter, $\log_v$, is  the inverse of the exponential mapping
$\exp_v:T_{v}V\to V$. By assumptions,
it is well defined on balls of the radius $\delta^{-1}\gg 1$.

Denote by $\tau_{v_1,v_2}\colon T_{v_1}V\to T_{v_2}V$ the parallel
translation  along the minimal geodesic (which is always unique in
our conditions) joining  points $v_1,v_2\in V$. Let $a,\xi,\eta\in
T_vV$, with $\|a\|=r<2$,  and denote $v'=\exp_v a$. Let $\gamma$
be the geodesic connecting $v$ with $v'$,
$\gamma(t)=\exp_v(ta/r)$. Obviously, $d_a\exp_v(\xi)=J(r)$, where
$J$ is the Jacobi field along $\gamma$ with the initial values
$J(0)=0, \dot J(0)=\xi/r$.

 Denote by $\eta(t)$  the parallel translation of $\eta$ to
$\gamma(t)$ along $\gamma$, so $\dot \eta=0$, and denote
$f(t)=\langle\eta(t),J(t)\rangle$. Evidently, $f(0)=0$, $\dot
f(0)=\langle\xi/r,\eta\rangle$, and
$$
|\ddot f(t)|=\langle\eta(t),\ddot J(t)\rangle=\langle\eta(t),
R(J,\dot\gamma)\dot\gamma\rangle
\le \delta \ \|\eta\|\ \|\xi\|\left(1+\frac \delta 2 t^2\right)
\le 2 \delta \ \|\eta\|\ \|\xi\|
$$
by Theorem~\ref{thm:rauch original}. This implies that
$$
|f(r) -\langle\xi,\eta\rangle|\le\delta r^2\ \|\eta\|\ \|\xi\|.
$$
Therefore, $\|d_a\exp_v(\xi)-\tau_{v,v'}\xi\|\le \delta r^2\|\xi\|$,
which proves the following Lemma.

\begin{Lem}\label{lem:Rauch} Let $a\in T_vV$, $v'=\exp_v a$. If
$r=\dist(v,v')<2$, then the  parallel translation
$\tau_{v,v'}\colon T_vV\to T_{v'}V$ and the differential
$(d\exp_v)_a\colon T_vV\to T_{v'}V$ are $r^2\delta$-close one to
the other. In particular, $\exp_v$ is $(1+\delta
r^2)$-bi-Lipschitz  on balls $B_{r}(v)\in V$  of radius $r<2$.

Similar statements hold for $d\log_w$ and $\tau_{x,w}$.
\end{Lem}


\begin{Lem}\label{lem:karh}
Let $x,y,z\in B_2(v).$  Then
\begin{equation}\label{eq:karh}
\|\log_zy-\log_zx- \tau_{x,z}\log_xy\|<C\,\delta
\end{equation}
\end{Lem}
\begin{proof}

This lemma is a simple consequence of one Karcher's estimate
(\cite{Karcher}, inequality  (C2.3) on the page 540).

Indeed, substituting  $a=\log_zx,$
$v=\log_zy-\log_zx,$ $p=z$ to formula  (C2.3)\cite{Karcher}, and taking into account
that section curvatures $K_{\sigma}<\delta$, one gets
 $$
 d(y, \exp_x(\log_zy-\log_zx)\leq C\delta.
 $$
 Since logarithmic mapping has very small distortion (Lemma \ref{lem:Rauch}), we
 obtain desired inequality by replacing points to their
$\log_x$-images in the last inequality .

\end{proof}

\section{Local maps $\varphi_i(x)$}

\subsection{Construction of $\phi_i$}\label{ssec:construction phi}
Let $\{v_i\}$ be a $\varepsilon$-separated $\varepsilon$-net in
$V$.
By definition of the Gromov--Hausdorff
metric, there exists a
mapping $\chi:\{v_i\}\to W$ such that $U_\delta(Im\ \chi)=W$ and
\begin{equation}\label{eq:chi property}
|\dist_V(v_i,v_j)-\dist_W(\chi(v_i),\chi(v_j))|<\delta
\end{equation}
for all $i,j$.
We will call such mappings by $\varepsilon$-approximations.

The construction of mappings $\varphi_i$ is the same for all $i$, so
we choose some $v=v_i$ and drop the index $i$ till the end of the
subsection.

\begin{enumerate}
\item Choose $\mathfrak{B}=\{b_1,b_2,\dots,b_n\}\subset T_{v}V$ be
    some orthonormal basis of the Euclidean space $T_{v}V$.
\item Pick $e_k\in\{v_j\}$ such that $\dist(e_k,
    \exp_{v}b_k)<\e$ for $k=1,\dots,n$. Denote by $\mathfrak{E}$
    the basis $\{\log_v e_k\}$ of $T_vV$, and let $\mathfrak{F}$
    be the basis $\{\log_w f_k\}$ of $T_wW$, where
    $w=\chi(v),f_k=\chi(e_k)\in W$
\item Denote by $L$ the linear mapping $T_{v}V\rightarrow T_{w}W$ such that
    \begin{equation}\label{eq:def L}
    L(\log_{v}(e_k))=\log_{w}(f_k),    \quad k=1,\dots,n.
    \end{equation}
    In the other words, $L$ is the linear extension of the
restriction to the basis $\mathfrak{E}$ of the mapping
$\log_w\circ\chi\circ\exp_v:T_vV\to T_wW$ (the latter is defined
on a discreet set of points only).
\item We define mapping $\varphi$ as:
    \begin{equation}\label{eq:def of varphi}
    \varphi=(\exp_{w}\circ L\circ\log_{v})|_{B_{4\e}(v)},
    \end{equation}
    where $B_{4\e}(v)\subset V$ is a  ball of radius $4\e$ with
    center in $v$.
\end{enumerate}
Evidently, $L=d_v\varphi$. Also, $\varphi(e_k)=f_k$ for $k=1,\dots,n$.

\begin{Lem}\label{lem:phi=chi on vi}
The base $\mathfrak{E}$ and $\mathfrak{F}$ are $C\e$-orthonormal. The
mapping $L:T_{v}V\rightarrow T_{w}W$ is $C\delta$-close to an isometry.

Let $e'\in \{v_k\}\subset V$ be a point in the net $\{v_i\}$,
$\dist(e',v)<2$, and let $f'=\chi(e')$. Then
$\|L(\log_ve')-\log_wf'\|\le C\delta$.
\end{Lem}

\begin{proof}
The norms of  the vectors $\log_ve_k-b_k$ are less than $2\e$ by the
choice of $e_k$ and Lemma~\ref{lem:Rauch}. Therefore $\mathfrak{E}$
is $C\e$-orthonormal basis.

 \mitrii{I think the "both $L\circ \log_v$ and" should be
removed here: this is what the lemma proves after all}

The pairwise distances between $v, e_k, e'$ change by at most
$C\delta$ by $\log_w\circ \chi$, due to Lemma~\ref{lem:Rauch} and the
main property \eqref{eq:chi property} of $\chi$. This implies that
the scalar products $\langle\log_ve_k,\log_ve_l\rangle$ differ from
the scalar products $\langle\log_wf_k, \log_wf_l\rangle$ by at most
$C\delta$,
\begin{equation}\label{eq:basis scalar products}
|\langle\log_ve_k,\log_ve_l\rangle-\langle\log_wf_k,
\log_wf_l\rangle|<C\delta,
\end{equation}
due to the cosine theorem
\begin{equation}\label{eq:cosine thm}
2\langle a,b\rangle=\|a\|^2+\|b\|^2-\|a-b\|^2.
\end{equation}
By the same reasons,
\begin{equation}\label{eq:vector scalar products}
|\langle\log_ve',\log_ve_l\rangle-\langle\log_wf', \log_wf_l\rangle|<C\delta.
\end{equation}

The \eqref{eq:basis scalar products} means that the difference of the
Gram matrices $G_{\mathfrak{E}},G_{\mathfrak{F}}$ of the base
$\mathfrak{E}$ and $\mathfrak{F}$ correspondingly, is
$C\delta$-small,
$$\|G_{\mathfrak{E}}-G_{\mathfrak{E}}\|<C\delta,$$
and, by $C\e$-orthogonality of $\mathfrak{E}$, this implies that both matrices
are $C\e$-close to the identity matrix. This implies that the basis
$\mathfrak{F}$ is a $C\e$-orthonormal basis.

Also, it implies that their inverses
$G_{\mathfrak{E}}^{-1},G_{\mathfrak{F}}^{-1}$ are $C\delta$-close as
well, so $L$ is $C\delta$-close to an isometry.

Together with \eqref{eq:vector scalar products}, this implies
that the coefficients for decomposition of vectors $\log_ve',\log_wf'$
in bases $\mathfrak{E}$ and $\mathfrak{F}$ correspondingly are
$C\delta$-close, since one can restore these coefficients from the
tuples  $\{(\log_ve',\log_ve_l)\}$ and $\{(\log_wf', \log_wf_l)\}$
using $G_{\mathfrak{E}}^{-1},G_{\mathfrak{F}}^{-1}$ correspondingly.
Since $L\mathfrak{E}=\mathfrak{F}$ by definition, the coefficients of
decomposition of $L(\log_ve')$ and $\log_wf'$ in the basis
$\mathfrak{F}$ are $C\delta$-close, which proves the Lemma.
\end{proof}

\begin{Cor}\label{cor:phi are Lipshcitz}
There is some constant $C>0$ which depends on $n$ only, such that
for each $i$, $\varphi_i$ is a $(1+C\delta)$-bi-Lipschitz map and
its differential is $C\delta$-close to isometry.
\end{Cor}
This follows immediately from the previous Lemma and
Lemma~\ref{lem:Rauch}.

\subsubsection{Maps $\{\phi_i\}$ are $C^1$-close one to another.}
Let $v_1,v_2\in\{v_i\}$ be two points of the $\e$-net on $V$, and
assume that $\dist(v_1,v_2)<4\e$. Here we prove that the mappings
$\phi_1$ and $\phi_2$ are $C\delta$-close in $C^1$-sense in
$B_{2\e}v_1\cap B_{2\e}v_2$.

\begin{Lem}\label{lem:phi_1=phi_2}
For every $x\in B_{4\varepsilon}(v_1)\bigcap B_{4\varepsilon}(v_2)$
\begin{equation}
\dist(\varphi_1(x),\varphi_2(x))<C\delta.
\end{equation}

Moreover, the parallel translation
$\tau_{\phi_1(x),\phi_2(x)}\circ d_x\phi_1$ of  $d_x\phi_1$  is
$C\delta$-close to $d_x\phi_2$.
\end{Lem}
\begin{proof}
By Lemma~\ref{lem:Rauch}:
\begin{equation}\label{eq:tLt=dphi}
\|\tau_{w_1,x}\circ L_1\circ \tau_{x,v_1}-d\phi(x)\|<C\delta
\end{equation}

\mitrii{I corrected the sign and the indices in the following
equation}
The vector $L_1(\tau_{x,v_1}\log_xe_i)$ is $C\delta$-close
to
$$
L_1(\log_{v_1}x-\log_{v_1}e_i)=\log_{w_1}\phi_1(x)-\log_{w_1}f_i,
$$
by Lemma~\ref{lem:karh}. The parallel translation of the right
part of the last equation to the point $\phi_1(x)$ by
$\tau_{w_1,\phi_1(x)}$ is $C\delta$-close to \mitrii{the same
here} $\log_{\phi_1(x)}{f_i}$, again by the same
Lemma~\ref{lem:karh}. Summing up, we get
$$
\|d_x\phi_1(\log_xe_i)-\log_{\phi_1(x)}f_i\|<C\delta,
$$
and similarly for $d_x\phi_2$.

Therefore by Lemma~\ref{lem:karh} the vectors
$$
\tau_{\phi_1(x),\phi_2(x)}\circ
d_x\phi_1(\log_xe_i)-d_x\phi_2(\log_xe_i), \quad i=1,\dots,n,
$$
are $C\delta$-close to $\log_{\phi_1(x)}\phi_2(x)$; i.e., to zero, if
the $\dist(\phi_1(x),\phi_2(x))<C\delta$. Since $\{\log_xe_i\}$ is a
$C\e$-orthonormal basis of $T_xV$, this means that the first claim of
the Lemma implies the second.

Now,  by Lemma~\ref{lem:phi=chi on vi},
\begin{equation}\label{eq:w2=phi(v2)}
\dist(\phi_1(v_2),\phi_2(v_2))<C\delta.
\end{equation}
Therefore, $ \|\tau_{\phi_1(v_2),w_2}\circ
d_{v_2}\phi_1-L_2\|<C\delta.$

The parallel translation along a geodesic triangle $\triangle
w_1w_2\phi_1(v_2)$ is  $C\delta$-close to identity, so, using
\eqref{eq:tLt=dphi}, we conclude that
$$
\|\tau_{w_1,w_2}\circ L_1\circ \tau_{v_2,v_1}-L_2\|<C\delta.
$$
Applying this to the vector $\log_{v_2}x\in T_{v_2}V$ and using
Lemma~\ref{lem:karh} and \eqref{eq:w2=phi(v2)} we get
$$
\|\log_{w_2}\phi_2(x)-\log_{w_2}\phi_1(x)\|<C\delta,
$$
which, by Lemma~\ref{lem:Rauch}, proves the first statement of the
Lemma as well.
\end{proof}

\section{Gluing together local mappings}\label{sec:Phi}

Here we glue the mappings $\varphi_i$ into one mapping $h:V\to W$
using the center of mass construction of Karcher.

Suppose  that a Riemannian manifold $V$ is covered by balls $B_i$ of
radius $2\varepsilon$, and that for every $i$ there exists a mapping
$\varphi_i:B_i\to W$ which is $\delta$-close to isometry. Assume
that the images of $\varphi_i$ cover $W$ and are $\delta$-close in
$C^1$ sense on intersections of their domains. We prove
that there exists a bi-Lipschitz mapping between $V$ and $W$ which is $C^1$
close to $\varphi_i$ on $B_i$.

\subsubsection{Partition of unity: construction and
estimates}\label{sec:construction of psi} \mitrii{I do not know what
to do with this section. I cannot find a reference}

 The
partition of unity $\{\psi_i(x)\}$ is constructed in a standard way.
We  need some estimates on the norms of their differentials, so we
repeat this classical construction here.

Let
$\Psi(r)$  be any non-negative monotonic and $C^{\infty}-$smooth function such that
it is equal to $1$ for $r\leq 1$ and $0$ for $r\geq 2$. Consider
functions
$\tilde\psi_i(x)=\Psi(\varepsilon^{-1}\left|xv_i\right|)$.
We define
$$\psi_i=\frac{\tilde\psi_i}{\sum\limits_{k}\tilde\psi_k}$$
It is easy to estimate the differential of $\psi_i$:

\begin{Lem}\label{lem:estim_deriv_psi}
For every $x\in V$
\begin{enumerate}
\item at least one of the numbers $\tilde\psi_i(x)$ is equal to 1,
\item at most $7^n$ of $\psi_i(x)$ are different from zero
    assuming that $\delta$ is sufficiently small,
\item $\parallel d_x\psi_i\parallel \le C\varepsilon^{-1}$ with some
    absolute constant $C$.
\end{enumerate}
\end{Lem}

\begin{proof}

The first statement follows because $\{v_i\}$ is an $\e$-net.

The Bishop-Gromov upper estimate for the number of non-vanishing
$\psi_i(x)$ is $\mbox{Vol}_{-\delta}B_{5\e/2}/\mbox{Vol}_{\delta}B_{\e/2}$,
where $\mbox{Vol}_cB_r$ denotes the volume of a ball of radius $r$ in the
space form of curvature $c$. For $\delta\ll 1$
this ratio is smaller than $7^n$, and the second claim follows.

Let us estimate the differentials:
 \begin{equation}
  \left\|d\psi_i\right\|\le
  \left\|\frac{d\tilde{\psi}_i}{\sum\tilde\psi_k}\right\|+
 \left\| \frac{\tilde\psi_i\,\sum d\tilde{\psi}_i}
 {(\sum\tilde\psi_k)^2}\right\|\le
  \left\|d\tilde{\psi}_i\right\|+
  \left\|\tilde\psi_i\cdot\sum d\tilde{\psi}_i\right\|,
  \end{equation}
because $\sum\limits_k\tilde\psi_k\ge 1$.

But $\parallel  \d\tilde\psi_k\parallel \le C{\varepsilon}^{-1}$,
where $C$ is some absolute constant. So the estimate for
$\parallel d\psi_i\parallel $ follows
from the second claim of the Lemma.
\end{proof}
\subsection{Definition of $\Phi$}
Define the function $\Phi:\ V\times W\rightarrow\R$ as
$$
\Phi(x,y)=\frac{1}{2}\sum_{i}\psi_i(x)
\dist(\varphi_i(x),y)^2,
$$
This is a smooth function because if $\dist(\varphi_i(x),y)^2$ is big,
than the corresponding
coefficient is zero.
We define the mapping $h(x)$
by the condition
$$\Phi(x,h(x))=\min_{y\in W}\Phi(x,y).$$
In the other words, $h(x)$ is the center of mass for the points
$\varphi_i(x)$ with weights $\psi_i(x)$.

\begin{Lem}[Karcher \cite{Karcher}]\label{lem:h is well-defined}
The function $h(x)$ is well-defined.
\end{Lem}

The reason is that for a fixed
$x\in V$, the points $\varphi_i(x)$ corresponding to non-zero
$\psi_i(x)$ are $C\delta$-close one to another by
Lemma~\ref{lem:phi_1=phi_2}. Therefore $\Phi(x,y)\ge 4C^2\delta^2$
for $y$ outside a ball $B_{3C\delta}\phi_i(x)\subset W$ of radius
$3C\delta$ centered at one of $\phi_i(x)$, and is less that
$C^2\delta^2$ for this $\varphi_i(x)$. Therefore the minimum lies in
$B_{3C\delta}\phi_i(x)$.
On the other hand, one can show that $\Phi(x,\exp(\cdot))$ is a convex function
in $B_{4\delta}(0)\subseteq T_{\phi_i(x)}W$, similar to Corollary~\ref{isom}
below, so the minimum is unique.
In particular, we see that
\begin{equation}\label{eq:h(x) close to phi(x)}
\dist(h(x),\phi_i(x))<3C\delta.
\end{equation}

Since for each $x\in V$, the function $\Phi$ reaches minimum at $y=h(x)$, we have
\begin{equation}\label{minimum}
 d_*\Phi=0,
 \end{equation}
at all points $(x,h(x))$, where
$d_*$ means the restriction of $d\Phi$ to the subspace $\{0\}\times
TW\subset T(V\times W)$. \mitrii{I replaced "In the next section "}
Below we will prove that the restriction $\mbox{Hess}_*\Phi$ of the
hessian of $\Phi$ to $TW\times TW$ is not degenerate at points
satisfying \eqref{minimum}. Then the function $h$ is smooth by the
implicit function theorem. Substituting $y=h(x)$ in \eqref{minimum}
and differentiating, we obtain

$$d^2\Phi (a,b)+d^2\Phi (dh(a), b)=0.$$

As a result
\begin{equation}\label{h-bi-Lipsch}
dh=d^2_*\Phi^{-1}\circ d^2_{**}\Phi,
\end{equation}
where $d^2_*\Phi$ and $d^2_{**}\Phi$ mean  the restrictions of
$d^2\Phi$ to $TW\times TW$ and $TV\times TW$,  accordingly,
understood as mappings $d^2_*\Phi: TW\to T^*W$ and
$d^2_{**}\Phi:TV\to T^*W$.

\begin{Lem}\label{lem:h is isometry}
The mapping $h:V\to W$ is smooth.

Moreover, its differential  is non-degenerate for each $x\in V$ and
has a bi-Lipschitz constant $\Delta(\delta)=O(\e)$, where
$\e^2=\delta$.
\end{Lem}

Theorem \ref{thm:main} easily follows from Lemma \ref{lem:h is isometry},
see Section \ref{last} below.

The first assertion of the lemma is already proved up to
non-degeneracy of $\mbox{Hess}_*\Phi$.

So to prove the lemma, it is enough to check that at points
$(x,h(x))$, both  $d^2_{**}\Phi$ and $d^2_*\Phi$ are $C\e$-close to
isometries.

Denote $\dist(x,y)$ by $|x,y|$. We claim that it is enough to prove
that, first, the similarly defined $d^2_{**}|\phi_i(x),y|:TV\to T^*W$
are all $C\e$-close and also $C\e$-close to an isometry, and, second,
that the same statements hold for $d^2_{*}|\phi_i(x),y|: TW\to T^*W$.

Indeed,  for $a\in T_xV$ we have $d_{**}^2\Phi(a)\in T^*W$, and
\begin{equation}\label{eq:Phixy}
d_{**}^2\Phi(a) =\sum_{i=1}^N d_x\psi_i(a)\cdot d_y|\varphi_i(x),y|+
\sum_{i=1}^N\psi_i(x)d^2_{**}|\varphi_i(x),y|(a),
\end{equation}
where $y=h(x)$. If all $d^2_{**}|\phi_i(x),y|$ are all $C\e$-close to
some isometry then the second sum, being their convex combination, is
also $C\e$-close to the same isometry. But the first sum in
\eqref{eq:Phixy} is $C\e$-small: by Lemma \ref{lem:estim_deriv_psi}
there are no more than $7^n$ non-zero terms in the first sum, and
$|d_x\psi_i(a)|\le C\delta^{-1/2}$, $|\varphi_i(x),h(x)|\le C\delta$
and $\|d_y|\varphi_i(x),h(x)|\|\leq 1$ in each of them.

For $d_{*}^2\Phi$ the proof is similar.

So Lemma \ref{lem:h is isometry} follows from the next lemma, and its
Corollary.

\begin{Lem} \label{second_diff}
Let $|x,y|$ be the distance function of a Riemannian manifold $M$.
Suppose, that points $x$, $y$ are joined by a unique geodesic
$\gamma\colon [0,1]\to M$ and no one point of $\gamma$ is
conjugate with $x$, $y$. Let $a\in T_yM$, $ b\in T_yM$, $\tilde b$
is obtained by parallel translation of $ b$ along $\gamma$ to $x$.
Then at the point $(x,y)\in M\times M$ we have:

(i) $d_*^2|x,y|(a)(a)=\langle J(1), J'(1)\rangle $, where $J$ is the
Jacobi field along $\gamma$ with initial data $J(0)=0$, $J(1)=a$.

(ii) $d_{**}^2|x,y|(a)(\tilde b)=\langle J(t), J'(t)\rangle|_0^1$,
where the Jacobi field $J$ satisfies the conditions
 $J(0)=\tilde b$, $J(1)=a$.
\end{Lem}

In the lemma, notations $d^2_*, \,\, d^2_{**}$ have the same sense as in formula
\eqref{h-bi-Lipsch}.

\begin{proof} In the case (i) $d_*^2|x,y|(a)(a)$ is equal to the second variation of
energy  for geodesic variation $\sigma$ of $\gamma$ such that $\sigma
(0, \tau)\equiv\gamma(0)$ and $\frac{D\sigma}{\d\tau}=a$, and
$\frac{D^2\sigma}{\d\tau^2}=0$ at $\gamma(1)$. It is well known that
this second variation is equal $\langle J(1), J'(1)\rangle $.

The case (ii) is proved similarly.
\end{proof}

\begin{Cor}\label{isom}
$d_*^2|x,y|(a)(a)$ and $d_{**}^2|x,y|(a)(\tilde b)$ are
$C\delta$-close to $\| a\|^2$ and $\langle a,b\rangle$,
correspondingly.
\end{Cor}
Indeed, in the case of the Euclidean metric $\langle J(1),
J'(1)\rangle$ is equal $a^2$ in the first case.  The Rauch theorem
shows that$\langle J(1), J'(1)\rangle$ is $C\delta$-close to $a^2$ in
our assumptions.

The second case is similar.

\begin{Cor}\label{isom2}
The mappings $d^2_{*}|\phi_i(x),y|: TW\to T^*W$ are all $C\e$-close
one to another and $C\e$-close to some isometry. The same holds
for $d^2_{**}|\phi_i(x),y|: TV\to T^*W$ .
\end{Cor}
Indeed, $d^2_{*}|\phi_i(x),y|$ are $C\delta$-close to $b\to \langle
\cdot, b\rangle$ by previous corollary, so the first claim is
trivial.

Similarly, $d^2_{**}|\phi_i(x),y|$ are $C\delta$-close to $b\to
\langle \cdot, d_x\phi_i(b)\rangle$, which are $C\delta$-close by
Lemma~\ref{lem:phi_1=phi_2} and $C\delta$ close to isometry by
Corollary~\ref{cor:phi are Lipshcitz}.

\section{Proof of Theorem~\ref{thm:main}}\label{last}

The differential of the mapping $h:V\to W$ defined by
Lemma~\ref{lem:h is well-defined} is $C\e$-close to isometry, in
particular non-degenerate by Lemma~\ref{lem:h is isometry}. This
means that $h(V)$ is open, and, as $V$ is compact, also closed.
Therefore $h(V)=W$, and $h$ is a covering.

Now, if $y=h(x_1)=h(x_2)$, then choosing $v_i$ such that
$\dist(v_i,x_i)<\e$, we get $\dist(h(v_i),y)<3\e/2$, which means that
$\dist(v_1,v_2)<2\e$, so $\dist(v_1,x_2)<3\e$.

The mapping  $h$ is smoothly homotopic to $\phi_1$ on the ball
$B_{7\e/2}(v_1)$: one can smoothly deform the partition of unity
in such a way that $\psi_1$ becomes identical $1$ on
$B_{7\e/2}(v_1)$. During this homotopy the mapping $h$ remains
smooth, and also $h(\d B_{7\e/2}(v_1))$ remains $C\delta$-close to
$\phi_1(\d B_{7\e/2}(v_1))$, by
Corollary~\ref{isom} and \eqref{eq:h(x) close to phi(x)} hold for any
partition of unity. This means that the number of preimages of
$y\in B_{2\e}(\chi(v_1))$ in $B_{7\e/2}(v_1)$ remains the same
during the homotopy. For $\phi_1$ this number is equal to $1$, as
$\phi_1$ is a diffeomorphism on $B_{4\e}(v_1))$, so $x_1=x_2$.

Therefore $h$ is a diffeomorphism, which, together with
Lemma~\ref{lem:h is isometry}, proves Theorem~\ref{thm:main}.





\begin{thebibliography}{1}

\bibitem{BBI}Burago, D. Yu.; Burago, Yu. D.; Ivanov, S. \emph{A course in metric
geometry. Graduate Studies in Mathematics}, 33. American
Mathematical Society, Providence, RI, 2001. xiv+415 pp. ISBN:
0-8218-2129-6

\bibitem{Buzalg}Burago, Yu. D.; Zalgaller, V. A. \emph{An introduction
    to Riemannian geometry\/} (Russian).  ``Nauka'', Moscow, 1994, 319 pp. ISBN:
    5-02-024606-9
\bibitem{chavel}Chavel, I. \emph{Riemannian geometry --- a modern
    introduction\/.} Cambridge Tracts in Mathematics, 108. Cambridge
    University Press, Cambridge, 1993. xii+386 pp. ISBN:
    0-521-43201-4;
\bibitem{Gromov}Gromov, M.\emph{Metric Structures for Riemannian
and Non-riemannian Spaces},
 1997, 379-383.
\bibitem{Jost}Jost, Ju.
\emph{Riemannian geometry and geometric analysis.} Third edition.
Universitext. Springer-Verlag, Berlin, 2002. xiv+532 pp. ISBN:
3-540-42627-2 53-02 (58-02)

\bibitem{Karcher}Karcher, H. \emph{Riemannian center of mass and
    mollifier smoothing\/.} Comm. Pure Appl. Math. \textbf{30}
    (1977), no. 5, 509-541.

\end{thebibliography}
\end{document}